\documentclass[11pt,reqno]{amsart}
\usepackage{amscd,amssymb,amsmath,amsthm}
\usepackage[arrow,matrix]{xy}
\usepackage{graphicx}
\usepackage{epstopdf}
\usepackage{color}
\usepackage{multicol}
\newtheorem{thm}{Theorem}
\newtheorem{defn}{Definition}
\newtheorem{lemma}{Lemma}
\newtheorem{pro}{Proposition}
\newtheorem{rk}{Remark}

\numberwithin{equation}{section} 

\begin{document}
\title [DYNAMICAL SYSTEM OF MOSQUITO POPULATION]
{ Dynamical system of a mosquito population with distinct birth-death rates}

\author{  Z.S. Boxonov, U.A. Rozikov}

\address{Z.\ S.\ Boxonov\\ Institute of mathematics, 81, Mirzo Ulug`bek str., 100170,
Tashkent, Uzbekistan.}
 \email {z.b.x.k@mail.ru}

\address{U.\ A.\ Rozikov \\ Institute of mathematics, 81, Mirzo Ulug`bek str., 100170,
Tashkent, Uzbekistan.} \email {u.rozikov@mathinst.uz}

\begin{abstract} We study the discrete-time dynamical systems of a model of
wild mosquito population with distinct birth (denoted by $\beta$) and death (denoted by $\mu$) rates. The case $\beta=\mu$ was considered in our previous work.
In this paper we prove that for $\beta<\mu$ the mosquito population will die and for $\beta>\mu$ the population will survive,  namely, the number of  the larvaes goes to infinite and the number of adults has finite limit ${\alpha\over \mu}$, where $\alpha>0$ is the maximum emergence rete.
\end{abstract}


 \maketitle

\section{Introduction}

In \cite{L}, \cite{A}, \cite{B}, \cite{J}, \cite{RV} and \cite{Rpd} (see also references therein)
several kind of mathematical models of mosquito population are studied.
This paper is a continuation of our paper \cite{RB}, where
following \cite{J.Li} it was considered a model of the mosquito population.
  In the model a wild mosquito population without the presence of
sterile mosquitoes is considered. For the simplified stage-structured
mosquito population,  due to
the fact that the first three stages in a mosquito's life cycle are aquatic, it was grouped the
three aquatic stages of mosquitoes into one class and divide the mosquito population
into only two classes, one of which consists of the first three stages that is
called the larvae, denoted by $x$, and one of which consists of all adults, denoted
by $y$.

The birth rate is the oviposition rate of adults denoted by
$\beta(\cdot)$; let the rate of emergence from larvae to adults be a
function of the larvae with the form of $\alpha(1-k(x))$,
where $\alpha>0$ is the maximum emergence rete, $0\leq
k(x)\leq 1$, with $k(0)=0, k'(x)>0$, and $\lim_{x\rightarrow
\infty}k(x)=1$, is the functional response due to
the intraspecific competition \cite{J}. Let the death rate
of larvae be $d_{0}+d_{1}x$, and
the death rate of adults be constant $\mu$. Then in the absence of sterile mosquitoes, and
in case (as in \cite{J}) $k(x)=\frac{x}{1+x}$,  $\beta(\cdot)=\beta$
the interactive dynamics for the wild mosquitoes are governed by the following system:
\begin{equation}\label{system}
\left\{%
\begin{array}{ll}
    \frac{dx}{dt}=\beta y-\frac{\alpha x}{1+x}-(d_{0}+d_{1}x)x,\\[3mm]
    \frac{dy}{dt}=\frac{\alpha x}{1+x}-\mu y.
\end{array}%
\right.\end{equation}
Denote
\begin{equation}\label{cont.thm}r_{0}=\frac{\alpha\beta}{(\alpha+d_{0})\mu}.
\end{equation}
Theorem 3.1 in \cite{J.Li} states  that:
\begin{itemize}
\item[-] if $r_{0}\leq1$ then the trivial equilibrium $(0;0)$ of system (\ref{system})
is a globally asymptotically stable, and there is no positive equilibrium.

\item[-] if $r_{0}>1$ then the trivial equilibrium $(0;0)$ is unstable, and there exists a unique positive equilibrium $(x_0, y_0)$ with
$$x_0=\frac{\sqrt{(d_{0}+d_{1})^{2}-4d_{1}(\alpha+d_{0})(1-r_{0})}-d_{0}-d_{1}}{2d_{1}}, \ \ y_{0}=\frac{\alpha x_0}{\mu(1+x_0)},$$ which is a globally asymptotically stable.
\end{itemize}

Define the operator $W:\mathbb{R}^{2}\rightarrow \mathbb{R}^{2}$ by
\begin{equation}\label{systema}
\left\{%
\begin{array}{ll}
    x'=\beta y-\frac{\alpha x}{1+x}-(d_{0}+d_{1}x)x+x,\\[3mm]
    y'=\frac{\alpha x}{1+x}-\mu y+y
\end{array}%
\right.\end{equation}
where $\alpha >0, \beta >0, \mu >0,\ d_{0}\geq0,\ d_{1}\geq0.$

In this paper we study the discrete time dynamical systems generated by (\ref{systema}).
In \cite{RB} for the evolution operator (\ref{systema}) the following results are obtained:

\begin{itemize}
\item[-] all fixed points of the evolution operator are found.
Depending on the parameters the operator may have unique, two and infinitely many fixed points;

\item[-] under some conditions on parameters type of each fixed point is determined and the
limit points of the dynamical system are given.

\item[-] for the case $\beta=\mu, \ \ d_{0}=d_{1}=0$  of parameters the full
analysis of corresponding dynamical system is given.\\
\end{itemize}
In this paper we consider the operator $W$ (defined by (\ref{systema})) for the case  $\beta\neq\mu, \ \ d_{0}=d_{1}=0$ and our
aim is to study trajectories of any initial point
from the invariant (under $W$) set $\mathbb{R}^2_{+}$.
The case when $d_{0}\ne 0$ or $d_{1}\ne 0$ is not studied yet.

\section{Dynamics for $\beta\neq\mu, \ \ d_{0}=d_{1}=0$.}
We assume
\begin{equation}\label{par}
\beta\neq\mu, \ \ d_{0}=d_{1}=0
\end{equation}
 then (\ref{systema}) has the following form
\begin{equation}\label{systemacase1}
W_{0}:\left\{%
\begin{array}{ll}
    x'=\beta y-\frac{\alpha x}{1+x}+x,\\[2mm]
    y'=\frac{\alpha x}{1+x}-\mu y+y.
\end{array}%
\right.\end{equation}
\begin{rk} Note that the operator $W_0$ well defined on $\mathbb{R}^2\setminus \{x=-1\}$. But to 
define a  dynamical system  of continuous operator and interpret values of $x$ and $y$ as probabilities we assume
$x\geq 0$ and $y\geq 0$. Therefore, we choose parameters of the operator $W_0$ to guarantee that it maps  $\mathbb{R}_{+}^{2}$ to itself.
\end{rk}

 It is easy to see that if
\begin{equation}\label{parametr}
0<\alpha\leq1, \ \ \beta >0, \ \ 0<\mu\leq1
\end{equation}
then operator (\ref{systemacase1}) maps $\mathbb{R}_{+}^{2}$ to itself.

A point $z\in\mathbb{R}_{+}^{2}$ is called a fixed point of $W_{0}$ if
$W_{0}(z)=z$.

For fixed point of $W_{0}$ the following holds.
\begin{pro}\label{lemmafix} The operator $W_0$ has a unique fixed point $z=(0; 0)$ in $\mathbb{R}_{+}^{2}$.
\end{pro}
\begin{proof} We need to solve
\begin{equation}\label{Fsystema}
\left\{%
\begin{array}{ll}
    x=\beta y-\frac{\alpha x}{1+x}+x,\\[3mm]
    y=\frac{\alpha x}{1+x}-\mu y+y.
\end{array}%
\right.\end{equation}
It is easy to see that $x=0, y=0.$
\end{proof}

Now we shall examine the type of the fixed point.
\begin{defn}\label{d1}
(see\cite{D}) A fixed point $s$ of an operator $W$ is called
hyperbolic if its Jacobian $J$ at $s$ has no eigenvalues on the
unit circle.
\end{defn}

\begin{defn}\label{d2}
(see\cite{D}) A hyperbolic fixed point $s$ called:

1) attracting if all the eigenvalues of the Jacobi matrix $J(s)$
are less than 1 in absolute value;

2) repelling if all the eigenvalues of the Jacobi matrix $J(s)$
are greater than 1 in absolute value;

3) a saddle otherwise.
\end{defn}
To find the type of a fixed point of the operator (\ref{systemacase1})
we write the Jacobi matrix:

$$J(z)=J_{W_{0}}=\left(%
\begin{array}{cc}
  1-\alpha & \beta \\
  \alpha & 1-\mu \\
\end{array}%
\right).$$
The eigenvalues of the Jacobi matrix are
$$\lambda_{1}=\frac{1}{2}\left(2-\alpha-\mu+ \sqrt{(\alpha-\mu)^2+4\alpha\beta}\right).$$
$$\lambda_{2}=\frac{1}{2}\left(2-\alpha-\mu- \sqrt{(\alpha-\mu)^2+4\alpha\beta}\right).$$

For type of $z=(0,0)$ the following lemma holds.
\begin{pro}\label{atr} The type of the fixed point $z$  for (\ref{systemacase1}) are as follows:
\begin{itemize}
  \item[i)] if\ $\beta<\mu$ then $z$ is attracting;
  \item[ii)]if\ $\beta>\mu$ then $z$ is saddle;
  \end{itemize}
\end{pro}
\begin{proof} For attractiveness we should have 
\begin{equation}\label{l12}
|\lambda_{1,2}|=\left|\frac{1}{2}\left(2-\alpha-\mu\pm \sqrt{(\alpha-\mu)^2+4\alpha\beta}\right)\right|<1.
\end{equation}
The inequality (\ref{l12}) is equivalent to the following
\begin{equation}\label{l12at}
\left\{%
\begin{array}{ll}
    \alpha+\mu+\sqrt{(\alpha-\mu)^2+4\alpha\beta}<4,\\[3mm]
    0<\alpha+\mu-\sqrt{(\alpha-\mu)^2+4\alpha\beta}<4.
\end{array}%
\right.\end{equation}
From the condition (\ref{parametr}) of the parameters $\alpha,\mu$ it follows that $\alpha+\mu\leq2$. 
If  $\beta<\mu$ then $\alpha+\mu>\sqrt{(\alpha-\mu)^2+4\alpha\beta}$ and the system of inequalities (\ref{l12at}) holds.
In the case $\beta>\mu$ we have $\lambda_1>1$ but $|\lambda_2|<1$ therefore the fixed point is a saddle point. 
\end{proof}

The following theorem  describes the trajectory of any point $(x^{(0)}, y^{(0)})$ in $\mathbb{R}^2_{+}$.
\begin{thm}\label{pr} For the operator $W_{0}$ given by (\ref{systemacase1}) (i.e. under condition (\ref{par})) 
and for any initial point $(x^{(0)}, y^{(0)})\in \mathbb R^2_+$ the following hold
$$\lim_{n\to \infty}x^{(n)}=\left\{\begin{array}{ll}
0, \ \ \ \ \ \ \mbox{if} \ \ \beta< \mu, \\[2mm]
+\infty, \ \ \mbox{if} \ \ \beta> \mu
\end{array}\right.$$
$$\lim_{n\to \infty}y^{(n)}=\left\{\begin{array}{ll}
0, \ \ \ \ \ \ \mbox{if} \ \ \beta< \mu, \\[2mm]
\frac{\alpha}{\mu}, \ \ \ \ \  \mbox{if} \ \ \beta> \mu
\end{array}\right.$$
where $(x^{(n)}, y^{(n)})=W_0^n(x^{(0)}, y^{(0)})$, with $W_{0}^n$ is $n$-th iteration of $W_{0}$.
\end{thm}
\begin{proof} \textbf{1)} Let $\beta<\mu.$ Then there exists $k>1$ such that $\beta\cdot k=\mu$.
Denote
$c^{(n)}=x^{(n)}+y^{(n)}$ and $c^{(n)}_{0}=k\cdot x^{(n)}+y^{(n)}$, where $x^{(n)}, y^{(n)}$ defined by the following
\begin{equation}\label{recc}\begin{array}{ll}
x^{(n)}=\beta y^{(n-1)}-\frac{\alpha x^{(n-1)}}{1+x^{(n-1)}}+x^{(n-1)},\\[3mm]
 y^{(n)}=\frac{\alpha x^{(n-1)}}{1+x^{(n-1)}}-\mu y^{(n-1)}+y^{(n-1)}.
 \end{array}\end{equation}
Both sequences $\{c^{(n)}\}$ and $\{c^{(n)}_{0}\}$ are monotone and bounded,
i.e.,
$$0\leq...\leq c^{(n)}\leq c^{(n-1)}\leq...\leq c^{(0)},$$
$$ 0\leq...\leq c^{(n)}_{0}\leq c^{(n-1)}_{0}\leq...\leq c^{(0)}_{0}.$$
Thus $\{c^{(n)}\}$ and $\{c^{(n)}_{0}\}$ have limit point, denote the limits by $c^*$ and $c^*_{0}$ respectively. 
Consequently, the following limits exist 
$$x^*=\lim_{n\to \infty}x^{(n)}=\frac{1}{1-k}\lim_{n\to \infty}(c^{(n)}-c^{(n)}_{0})=\frac{1}{1-k}(c^*-c^*_{0}),$$
$$y^*=\lim_{n\to \infty}y^{(n)}=c^*-x^*.$$
 and by (\ref{recc}) we have
$$x^*=\beta y^*-\frac{\alpha x^*}{1+x^*}+x^*,\ \ \ y^*=\frac{\alpha x^*}{1+x^*}-\mu y^*+y^*,$$
i.e., $x^*=0, y^*=0.$
\begin{itemize}
\item[\textbf{2)}] Let $\beta>\mu.$ The proof is based on the following three lemmas.
\end{itemize}
\begin{lemma}\label{border} For any parameters satisfying (\ref{parametr}) and for arbitrary 
initial point $(x^{(0)}, y^{(0)})$ in $\mathbb{R}^2_{+}$ the sequence $y^{(n)}$ (defined in (\ref{recc})) is bounded:
$$0\leq y^{(n)}\leq \left\{\begin{array}{ll}
y^{(0)}, \ \ \mbox{if} \ \ y^{(0)}>\frac{\alpha}{\mu}\\[2mm]
\frac{\alpha}{\mu}, \ \ \mbox{if} \ \ y^{(0)}<\frac{\alpha}{\mu}.
\end{array}\right.$$
\end{lemma}
\begin{proof} Note that on the set $\mathbb{R}_{+}$ the function $y(x)=\frac{x}{1+x}$ is increasing and
 bounded by 1. Therefore
 $$y^{(n)}=\frac{\alpha x^{(n-1)}}{1+x^{(n-1)}}+(1-\mu)y^{(n-1)}\leq\alpha+(1-\mu)y^{(n-1)}\leq\alpha+(1-\mu)(\alpha+(1-\mu)y^{(n-2)})\leq$$
 $$\alpha+\alpha(1-\mu)+(1-\mu)^2(\alpha+(1-\mu)y^{(n-3)})\leq
 ...\leq\alpha+\alpha(1-\mu)+\alpha(1-\mu)^2+...+\alpha(1-\mu)^{n-1}$$ $$+(1-\mu)^{n}y^{(0)}=\frac{\alpha}{\mu}+(1-\mu)^{n}(y^{(0)}-\frac{\alpha}{\mu}).$$

 Thus if the initial point $y^{(0)}>\frac{\alpha}{\mu}$ then for any $m\in\mathbb{N}$ we have $0\leq y^{(m)}\leq y^{(0)}$. Moreover, if $y^{(0)}<\frac{\alpha}{\mu}$ then for any $m\in\mathbb{N}$ we have $0\leq y^{(m)}\leq \frac{\alpha}{\mu}.$
\end{proof}

\begin{lemma}\label{sequence} For sequences $x^{(n)}$ and $y^{(n)}$ the following statements hold:
\begin{itemize}
  \item[1)] For any $n\in\mathbb{N}$ and $\beta>\mu$ the inequalities
   $x^{(n)}>x^{(n+1)}$ and $y^{(n)}>y^{(n+1)}$ can not be satisfied at the same time.
  \item[2)] If  $x^{(m-1)}<x^{(m)}$, $y^{(m-1)}<y^{(m)}$ for some $m\in\mathbb{N}$ then $x^{(m)}<x^{(m+1)}$, $y^{(m)}<y^{(m+1)}$.
  \item[3)] If  $x^{(m-1)}>x^{(m)}$, $y^{(m-1)}<y^{(m)}$ for any $m\in\mathbb{N}$ then for  $\beta>\mu$  the inequalities $x^{(m)}>x^{(m+1)}$, $y^{(m)}<y^{(m+1)}$ can not be satisfied at the same time.
  \item[4)] If $x^{(m-1)}<x^{(m)}$, $y^{(m-1)}>y^{(m)}$ for any $m\in\mathbb{N}$  then for $\beta>\mu$ the inequalities $x^{(m)}<x^{(m+1)}$, $y^{(m)}>y^{(m+1)}$ can not be satisfied at the same time.
 \item[5)] For any $m\in\mathbb{N}$ the inequalities $x^{(m-1)}<x^{(m)}, x^{(m)}>x^{(m+1)}, y^{(m-1)}>y^{(m)}, y^{(m)}<y^{(m+1)}$  can not be satisfied at the same time.
\end{itemize}
\end{lemma}
\begin{proof} Adding $x^{(n)}$ and $y^{(n)}$ we get
\begin{equation}\label{rec}
x^{(n)}+y^{(n)}=(\beta-\mu)y^{n-1}+x^{(n-1)}+y^{(n-1)}
\end{equation}
\begin{itemize}
  \item[1)] From (\ref{rec}) by $\beta>\mu$ we get $(x^{(n)}-x^{(n-1)})+(y^{(n)}-y^{(n-1)})>0.$
  Consequently, $x^{(n)}>x^{(n+1)}$ and $y^{(n)}>y^{(n+1)}$ can not be satisfied at the same time.
  \item[2)] By $x^{(m)}-x^{(m-1)}>0$ we have
  $$x^{(m)}(1-\frac{\alpha}{1+x^{(m)}})-x^{(m-1)}(1-\frac{\alpha}{1+x^{(m-1)}})>0$$ and
  $$\frac{x^{(m)}}{1+x^{(m)}}>\frac{x^{(m-1)}}{1+x^{(m-1)}}.$$
 Then
 $$x^{(m+1)}-x^{(m)}=\beta(y^{(m)}-y^{(m-1)})+x^{(m)}(1-\frac{\alpha}{1+x^{(m)}})-x^{(m-1)}(1-\frac{\alpha}{1+x^{(m-1)}})>0,$$ $$y^{(m+1)}-y^{(m)}=\alpha(\frac{x^{(m)}}{1+x^{(m)}}-\frac{x^{(m-1)}}{1+x^{(m-1)}})+(1-\mu)(y^{(m)}-y^{(m-1)})>0.$$
  \item[3)] Assume  in case when $x^{(m-1)}>x^{(m)}$, $y^{(m-1)}<y^{(m)}$ hold for any $m\in\mathbb{N}$ then  for  $\beta>\mu$ the inequalities  $x^{(m)}>x^{(m+1)}$, $y^{(m)}<y^{(m+1)}$ are satisfied at the same time.
      Then since $x^{(n)}$ is decreasing and bounded; $y^{(n)}$ is increasing and bounded (see Lemma \ref{border}) there exist their limits $x^*$, $y^*\neq0$ respectively. By (\ref{recc}) we obtain
$$\left\{
  \begin{array}{ll}
   \beta y^* =\frac{\alpha x^*}{1+x^*}\\[2mm]
   \frac{\alpha x^*}{1+x^*}=\mu y^*\\
  \end{array}
\right.$$ i.e. $x^*=0, y^*=0.$ This contradiction shows that if for any $m\in\mathbb{N}$ one has  $x^{(m-1)}>x^{(m)}$, $y^{(m-1)}<y^{(m)}$ then for  $\beta>\mu$ the inequalities $x^{(m)}>x^{(m+1)}$, $y^{(m)}<y^{(m+1)}$ can not be satisfied at the same time (see Fig \ref{fig.1}).
    \item[4)] Assume if for any $m\in\mathbb{N}$ the inequalities  $x^{(m-1)}<x^{(m)}$, $y^{(m-1)}>y^{(m)}$ hold then for  $\beta>\mu$ the inequalities $x^{(m)}<x^{(m+1)}$, $y^{(m)}>y^{(m+1)}$ are satisfied at the same time, i.e. $x^{(m)}$ is increasing and $y^{(m)}$ is decreasing.
       Let
       $$x^{(m+1)}-x^{(m)}=\Delta^{(m)}, \ \ y^{(m+1)}-y^{(m)}=-\delta^{(n)},$$
       $$(x^{(m+1)}-x^{(m)})+(y^{(m+1)}-y^{(m)})=\Delta^{(m)}-\delta^{(m)}<(\beta-\mu)y^{(0)}.$$

Since $\{\Delta^{(m)}\}$ is decreasing, $\{\delta^{(m)}\}$ is increasing and  $\Delta^{(m)}-\delta^{(m)}>0$ for $\beta>\mu$ we conclude that the sequence   $\{\Delta^{(m)}-\delta^{(n)}\}$ is decreasing and bounded from below. Thus $\{\Delta^{(m)}-\delta^{(m)}\}$ has a limit and since $y^{(m)}$ has limit we conclude that $x^{(m)}$ has a finite limit. By (\ref{recc}) we have
$$\lim_{m\rightarrow\infty}x^{(m)}=0, \ \ \lim_{m\rightarrow\infty}y^{(m)}=0.$$
But this is a contradiction to $\lim_{m\rightarrow\infty}x^{(m)}\neq0$. This completes proof of part 4 (see Fig \ref{fig.2}).
    \item[5)] Assume for any $m\in\mathbb{N}$ one has  $x^{(m-1)}<x^{(m)}, x^{(m)}>x^{(m+1)}, y^{(m-1)}>y^{(m)}, y^{(m)}<y^{(m+1)}$. Then $x^{(m-1)}<x^{(m+1)}, y^{(m-1)}>y^{(m+1)}$. Moreover, for each $k\in\mathbb{N}$ for $m=2k$ we have $x^{(2k-1)}<x^{(2k+1)}, y^{(2k-1)}>y^{(2k+1)}$, and for  $m=2k+1$ we have  $x^{(2k)}<x^{(2k+2)}, y^{(2k)}>y^{(2k+2)}$. But by Lemma \ref{sequence} parts 3) and 4) these inequalities do not hold for any $k\in\mathbb{N}$ (see Fig \ref{fig.3}).
\end{itemize}
\end{proof}

\begin{figure}[h!]
\includegraphics[width=0.6\textwidth]{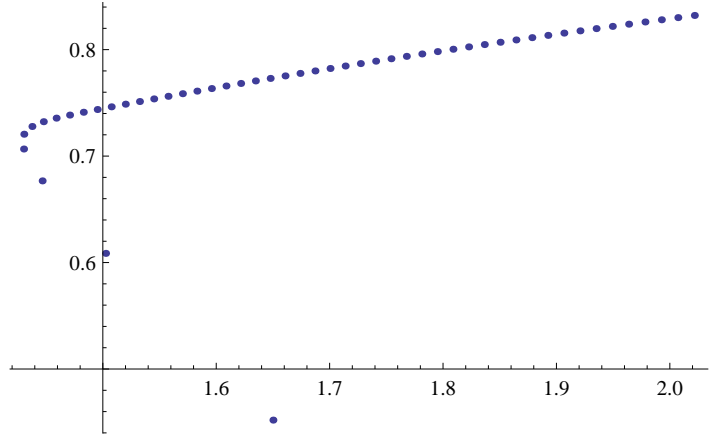}\\
\caption{$\alpha=0.6, \beta=0.5, \mu=0.48, x^{(0)}=2, y^{(0)}=0.1$}\label{fig.1}
\end{figure}

\begin{figure}[h!]
\includegraphics[width=0.6\textwidth]{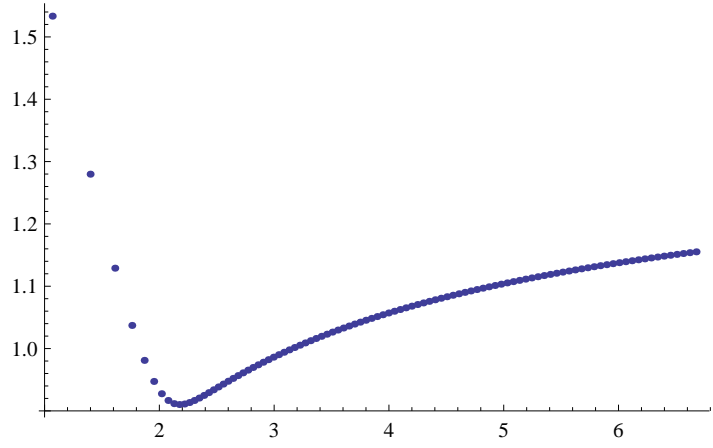}\\
\caption{$\alpha=0.4, \beta=0.35, \mu=0.3, x^{(0)}=0.5, y^{(0)}=2$}\label{fig.2}
\end{figure}

\begin{figure}[h!]
  \includegraphics[width=0.6\textwidth]{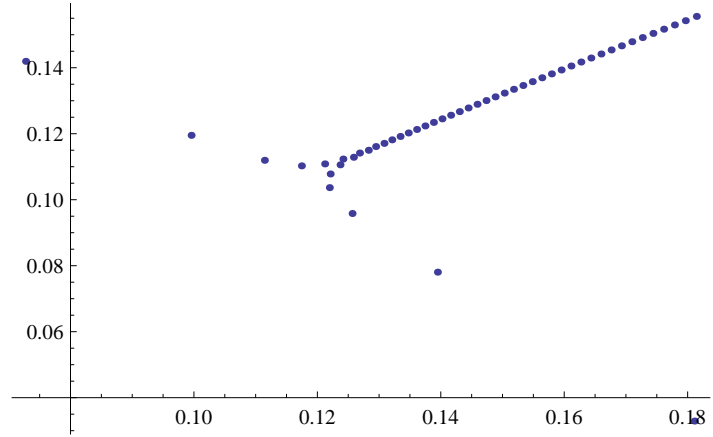}\\
  \caption{$\alpha=0.9, \beta=0.9, \mu=0.88, x^{(0)}=0.01, y^{(0)}=0.2$}\label{fig.3}
\end{figure}

\begin{lemma} There exists $n_0$ such that the sequences $x^{(n)}$ and $y^{(n)}$  are increasing for $n\geq n_0$ and $x^{(n)}$ is unbounded from above.
\end{lemma}
\begin{proof} Monotonicity of $x^{(n)}$ and $y^{(n)}$  follow from Lemma \ref{sequence}.
Consider
\begin{equation}\label{xsys}
\left\{
  \begin{array}{ll}
   (x^{(n_{0}+1)}-x^{(n_{0})})+(y^{(n_{0}+1)}-y^{(n_{0})})=(\beta-\mu)y^{(n_{0})} \\[2mm]
   (x^{(n_{0}+2)}-x^{(n_{0}+1)})+(y^{(n_{0}+2)}-y^{(n_{0}+1)})=(\beta-\mu)y^{(n_{0}+1)} \\
     ... \\
   (x^{(n-1)}-x^{(n-2)})+(y^{(n-1)}-y^{(n-2)})=(\beta-\mu)y^{(n-2)}\\[2mm]
   (x^{(n)}-x^{(n-1)})+(y^{(n)}-y^{(n-1)})=(\beta-\mu)y^{(n-1)}
  \end{array}
\right.
\end{equation}
Adding equations of (\ref{xsys}) we get
\begin{equation}\label{xsyst}
(x^{(n)}-x^{(n_{0})})+(y^{(n)}-y^{(n_{0})})=(\beta-\mu)(y^{(n_{0})}+y^{(n_{0}+1)}+...+y^{(n-2)}+y^{(n-1)})
\end{equation}
Let $y^{(n)}$ (see Lemma \ref{border}) is bounded by $\theta$.
By (\ref{xsyst}) we have $$x^{(n)}=x^{(n_{0})}+y^{(n_{0})}-y^{(n)}+(\beta-\mu)(y^{(n_{0})}+y^{(n_{0}+1)}+...+y^{(n-2)}+y^{(n-1)})$$ $$
>x^{(n_{0})}+y^{(n_{0})}-\theta+(\beta-
\mu)(n-n_{0})y^{(n_{0})}.$$

For $\beta>\mu$ from $$\lim_{n\rightarrow\infty}(x^{(n_{0})}+y^{(n_{0})}-\theta+(\beta-
\mu)(n-n_{0})y^{(n_{0})})=+\infty$$ it follows that  $x^{(n)}$ is not bounded from above.
\end{proof}
Thus for $\beta>\mu$ the sequence $y^{(n)}$ has limit $y^*$  (see Lemma \ref{border}).
 Consequently, by (\ref{recc}) and $\lim_{n\rightarrow\infty}x^{(n)}=+\infty $ we get $y^*=\frac{\alpha}{\mu}$.
 Theorem is proved.
\end{proof}

\emph{Biological interpretation} of our result is clear:  for $\beta<\mu$ the 
mosquito population will die and for $\beta>\mu$ the population will survive. Comparing our results of Theorem \ref{pr} with results of the continuous time dynamics (mentioned in the Introduction) one can see that they are the same.\\

A point $z$ in $W_{0}$ is called periodic point of $W_{0}$ if there exists $p$ so that $W_{0}^{p}(z)=z$. The smallest positive integer $p$ satisfying $W_{0}^{p}(z)=z$ is called the prime period or least period of the point $z.$

\begin{thm} For $p\geq 2$ the operator (\ref{systemacase1}) does not have any $p$-periodic point in the set $\mathbb{R}^{2}_{+}.$
\end{thm}
\begin{proof} This is a corollary of Theorem \ref{pr}. Here we give an alternative proof.
Let us first describe periodic points with $p=2$ on $\mathbb{R}^{2}_{+},$ in this case the equation $W_{0}(W_{0}(z))=z.$
That is
\begin{equation}\label{per.22}
\left\{
  \begin{array}{ll}
   x=\beta(\frac{\alpha x}{1+x}-\mu y+y)+(\beta y-\frac{\alpha x}{1+x}+x)(1-\frac{\alpha}{1+\beta y-\frac{\alpha x}{1+x}+x}), \\[3mm]
   y=\frac{\alpha(\beta y-\frac{\alpha x}{1+x}+x)}{1+\beta y-\frac{\alpha x}{1+x}+x}+(1-\mu)(\frac{\alpha x}{1+x}-\mu y+y).\\
  \end{array}
\right.
\end{equation}
Simple calculations show that the last equation is equivalent to the following
\begin{equation}\label{per} \frac{\alpha x}{1+x}=(\mu-2)y, \ \ \frac{\alpha x}{1+x}\geq0, \ \ (\mu-2)y\leq0.
\end{equation}

The only solution to (\ref{per}) is $x=0, y=0$. Thus the operator (\ref{systemacase1}) does not
have any two periodic point in the set $\mathbb{R}^{2}_{+}.$

Now we show that $W_0$ does not have any periodic point (except fixed).
Rewrite operator (\ref{systemacase1}) in normalized form:
\begin{equation}\label{systemnormal}
U:\left\{%
\begin{array}{ll}
    x'=\frac{(1+x)(x+\beta y)-\alpha x}{(1+x)(x+(\beta-\mu+1)y)},\\[2mm]
    y'=\frac{\alpha x+(1+x)(1-\mu)y}{(1+x)(x+(\beta-\mu+1)y)}.
\end{array}%
\right.\end{equation}
Denote $$S=\{(x,y)\in\mathbb{R}^2_{+}:x+y=1\}.$$
 From conditions (\ref{parametr}) one gets $x'\geq0$ and $y'\geq0$ moreover $x'+y'=1.$
Hence $U:S\rightarrow S.$
\begin{rk} The operators which map $S$ to itself have been extensively studied (see for example \cite{MSQ}, \cite{MS}, \cite{Rpd} and the references therein).
\end{rk}
Using $x+y=1$, from (\ref{systemnormal}) we get
\begin{equation}\label{T}
T:x'=\frac{(1-\beta)x^2+(1-\alpha)x+\beta}{(\mu-\beta)x^2+x+\beta-\mu+1}.
\end{equation}
Denote $$S^*=[0,1].$$
\begin{pro} If conditions (\ref{parametr}) are satisfied then the function $T$ (defined by (\ref{T}))  maps $S^*$ to itself.
\end{pro}
\begin{proof} We want to show that if $x\in S^*$ then $x'=T(x)\in S^*.$
Let $$a=(1-\beta)x^2+(1-\alpha)x+\beta,\ b=(\mu-\beta)x^2+x+\beta-\mu+1$$
Then $h(x)=b-a=(\mu-1)x^2+\alpha x+1-\mu.$
Since $h(0)=1-\mu>0$ and $h(1)=\alpha>0$ for each $x\in[0,1]$ there is $h(x)\geq0,$ i.e., $x'\leq1.$ By
$$a=(1-\beta)x^2+(1-\alpha)x+\beta=(1+x)(x+\beta (1-x))-\alpha x\geq0,$$ $$b=(\mu-\beta)x^2+x+\beta-\mu+1=(1+x)(x+(\beta-\mu+1)(1-x))>0$$
we have $x'\geq0,\ y'\geq0.$ Therefore under conditions (\ref{parametr}) we get $T:S^*\mapsto S^*$.
\end{proof}

Let us first describe periodic points of $U$ with $p=2$. In this case the equation $U(U(z))=z$
can be reduced to description of 2-periodic points of the function $T$ defined in (\ref{T}), i.e.,to solution of the equation
\begin{equation}\label{per.2}
T(T(x))=x.\end{equation}
Note that the fixed points of $T$ are solutions to (\ref{per.2}), to find other solution we consider the equation
$$\frac{T(T(x))-x}{T(x)-x}=0,$$
simple calculations show that the last equation is equivalent to the following
\begin{equation}\label{PER}
Ax^2+Bx+C=0.\end{equation}
where $$A=(1-\beta)(\beta-2)+(\beta-\mu+1)(\beta-\mu),$$ $$B=(\beta-2)(\beta-\mu-\alpha+2)-\beta(\beta-\mu),$$ $$C=(\beta-\mu+1)(\alpha+\mu-\beta-2)+\beta(\beta-1).$$
By (\ref{parametr}) we have $A+B+C<0$, $B<0$, $C<0$. Therefore since $x\geq 0$ the LHS of (\ref{PER}) is $<0$.
Consequently, the equation (\ref{PER}) does not have solution in $S^*$.
Thus function (\ref{T}) does not have any 2-periodic point in $S^*$.
Since $T$ is continuous on $S^*$ by Sharkovskii's theorem (\cite{D}, \cite{UR}) we have that
$T^p(x)=x$ does not have solution (except fixed) for all $p\geq 2$.
Hence it follows that the operator (\ref{systemacase1}) has no periodic points (except fixed) in the set $\mathbb{R}^2_{+}.$
\end{proof}

 \section*{Acknowledgements}

We thank both referees for their useful comments.

\end{document}